\documentclass{jvarticle}

\usepackage[british]{babel}
\usepackage{amsfonts,amsmath}

\usepackage[round]{natbib}

\usepackage{tikz}

\newcommand{\CD}{\mathcal{D}}
\newcommand{\CF}{\mathcal{F}}
\newcommand{\CH}{\mathcal{H}}
\newcommand{\CL}{\mathcal{L}}
\newcommand{\CN}{\mathcal{N}}
\newcommand{\CO}{\mathcal{O}}
\DeclareMathOperator{\Cov}{Cov}
\let\d=\partial

\newcommand{\DX}{\Delta x}
\def\e{\mathrm{e}}
\def\E{\mathbb{E}}
\let\eps=\varepsilon
\def\exact{^\mathrm{exact}}
\let\la=\langle
\newcommand{\N}{\mathbb{N}}
\let\Om=\Omega
\renewcommand{\P}{\mathbb{P}}
\let\phi=\varphi
\def\quark{\setbox0\hbox{$x$}\hbox to\wd0{\hss$\cdot$\hss}}
\let\ra=\rangle
\def\rC{\mathrm{C}}
\def\rH{\mathrm{H}}
\def\rL{\mathrm{L}}
\def\R{\mathbb{R}}
\def\SLR{\sigma_{LR}}
\def\SL{\sigma_{L}}
\def\SR{\sigma_{R}}
\DeclareMathOperator{\Var}{Var}
\def\Z{\mathbb{Z}}

\newenvironment{jvlist}{\begin{list}{}{%
      \setlength{\leftmargin}{2em}
      \setlength{\labelwidth}{1em}
      \setlength{\labelsep}{0.5em}
      \setlength{\topsep}{0pt}
      \setlength{\parsep}{\parskip}
      \setlength{\itemsep}{0pt}}}{\end{list}}

\hypersetup{pdftitle={The Effect of Finite Element Discretisation on
    the Stationary Distribution of SPDEs},
  pdfauthor={Jochen Voss}}

\begin{document}

\title{The Effect of Finite Element Discretisation on the
  Stationary Distribution of SPDEs}
\author{Jochen Voss}
\date{20th October 2011}
\maketitle

\begin{abstract}
  This article studies the effect of discretisation error on the
  stationary distribution of stochastic partial differential equations
  (SPDEs).  We restrict the analysis to the effect of space
  discretisation, performed by finite element schemes.  The main
  result is that under appropriate assumptions the stationary
  distribution of the finite element discretisation converges in total
  variation norm to the stationary distribution of the full SPDE.
\end{abstract}

\bigskip

\noindent
\textbf{keywords:} SPDEs, finite element discretisation, stationary
distribution


\section*{Introduction}

In this article we consider the finite element discretisation for
stochastic partial differential equations (SPDEs) of the form
\begin{equation}\label{E:simple}
  \d_t u(t,x)
    = \d_x^2 u(t,x)
      + f\bigl(u(t,x)\bigr)
      + \sqrt{2}\, \d_t w(t,x)
  \qquad \forall (t,x) \in [0,\infty)\times [0, 1],
\end{equation}
where $\d_t w$ is space-time white noise and $f\colon \R\to \R$ is a
smooth function with bounded derivatives, and the differential
operator~$\d_x^2$ is equipped with boundary conditions such that it is
a negative operator on the space $\rL^2\bigl([0,1], \R\bigr)$.  More
specifically, we are considering the effect that discretisation of the
SPDE has on its stationary distribution.

Our motivation for studying this problem lies in a recently
proposed, SPDE-based sampling technique: when trying to sample from a
distribution on path-space, \textit{e.g.}\ in filtering/smoothing
problems to sample from the conditional distribution of a process
given some observations, one can do so using a Markov chain Monte
Carlo approach.  Such MCMC methods requires a process with values in
path-space and it transpires that in some situations SPDEs of the
form~\eqref{E:simple} can be used, see \textit{e.g.}\
\citet{HaiStuaVoWi05,HaiStuaVo07} and \citet{HaiStuaVo09} for a
review.  When implementing the resulting methods on a computer, the
sampling SPDEs must be discretised and, because MCMC methods use the
sampling process only as a source of samples from its stationary
distribution, the effect of the discretisation error on an MCMC method
depends on how well the stationary distribution of the SPDE is
approximated.  While there are many results of approximation of
trajectories of SPDEs
\citep[\textit{e.g.}][]{MR2147242,Wa05,Hau08,MR2465711,Jen11}, approximation
of the stationary distribution seems not to be well-studied so far.

When discretising an SPDE, discretisation of space and time can be
considered to be two independent problems.  In cases where only the
stationary distribution of the process is of interest, Metropolis
sampling, using the next time step of the time discretisation as a
proposal, can be used to completely eliminate the error introduced by
time discretisation \citep{BeRoStuaVo08}.  For this reason, in this
article we restrict the analysis to the effect of space discretisation
alone.  The discretisation technique discussed here is a finite
element discretisation, which is a much-studied technique for
deterministic PDEs.  The approximation problem for {\em stochastic}
PDEs, as studied in this article, differs from the deterministic case
significantly, since here we have to compare the full {\em
  distribution} of the solutions instead of considering the
approximation of the solution as a function.

While the results of this article are formulated for SPDEs with values
in $\R$, we expect the results and techniques to carry over to SPDEs
with values in $\R^d, d>1$ without significant changes.  We only restrict
discussion to the one-dimensional case to ease notation.
This is in contrast to the domain of the SPDEs: we
consider the case of one spacial dimension because this is the
relevant case for the sampling techniques discussed above, but this choice
significantly affects the proofs and a different approach
would likely be required to study the case of higher-dimensional
spatial domains.

\bigskip

The text is structured as follows: In section~\ref{S:SPDE} present the
required results to characterise the stationary distribution of the
SPDE~\eqref{E:simple}.  In section~\ref{S:SDE} we introduce the finite
element discretisation scheme for~\eqref{E:simple} and identify the
stationary distribution of the discretised equation.  Building on
these results, in section~\ref{S:result}, we state our main result
about convergence of the discretised stationary distributions to the
full stationary distribution.  Finally, in section~\ref{S:example}, we
give two examples in order to illustrate the link to the MCMC methods
discussed above and also to demonstrate that the considered finite
element discretisation forms a concrete and easily implemented
numerical scheme.

\section{The Infinite-Dimensional Equation}
\label{S:SPDE}

In order to study the SPDE~\eqref{E:simple}, it is convenient to
rewrite the equation as an evolution equation on the Hilbert space
$\CH = \rL^2\bigl([0,2\pi],\R^d\bigr)$; For a description of the
underlying theory we refer, for example, to the monograph of
\citet{ZDP}.  We consider
\begin{equation}\label{E:SPDE}
  du(t) = \CL u(t) \,dt + f\bigl(u(t)\bigr) + \sqrt{2}\, dw(t)
  \qquad \forall t \geq 0
\end{equation}
where the solution $u$ takes values in $\CH$ and $f$ acts pointwise on
$u$, \textit{i.e.}\ $f(u)(x) = \tilde f\bigl(u(x)\bigr)$ for almost
all $x\in[0,1]$ for some function $\tilde f\colon \R^d \to \R^d$, such
that $f$ maps $\CH$ into itself.  Furthermore, $w$ is an
$\rL^2$-cylindrical Wiener process and we equip the linear operator $\CL
= \d_x^2$ with boundary conditions given by the domain
\begin{equation}\label{E:domain}
  \CD(\CL)
  = \bigl\{ u\in \rH^2([0,1], \R) \bigm|
                \alpha_0 u(0) - \beta_0 \d_x u(0) = 0, \;
                \alpha_1 u(1) + \beta_1 \d_x u(1) = 0 \bigr\}
\end{equation}
where $\alpha_0, \alpha_1, \beta_0, \beta_1\in\R$.  The boundary
conditions in~\eqref{E:domain} include the cases of Dirichlet
($\beta_i=0$) and v.~Neumann ($\alpha_i=0$) boundary conditions.  The
general case of $\alpha_i, \beta_i \neq 0$ is known as Robin
boundary conditions.

We start our analysis by considering the linear equation
\begin{equation}\label{E:linear}
  du(t) = \CL u(t) \,dt + \sqrt{2}\, dw(t) \qquad \forall t \geq 0.
\end{equation}
For equation~\eqref{E:linear} to have a stationary distribution, we
require $\CL$ to be negative definite.  The following lemma states
necessary and sufficient conditions on $\alpha_i$ and $\beta_i$ for
this to be the case.

\begin{lemma}\label{L:cond-neg}
  The operator $\CL$ is a self-adjoint operator on the Hilbert space
  $\CH$.  The operator $\CL$ is negative definite, if
  and only if $\alpha_0, \beta_0, \alpha_1, \beta_1$ are
  contained in the set
  \begin{multline*}
    A = \Bigl\{
      \beta_0(\alpha_0+\beta_0) >0,
      \beta_1(\alpha_1+\beta_1) >0,
      \big| (\alpha_0+\beta_0)(\alpha_1+\beta_1)\bigr| > | \beta_0\beta_1| \Bigr\} \\
    \cup \Bigl\{
      \beta_0 = 0, \alpha_0 \neq 0, \;
      \beta_1(\alpha_1+\beta_1) >0 \Bigr\}
    \cup \Bigl\{
      \beta_0(\alpha_0+\beta_0) >0, \;
      \beta_1 = 0, \alpha_1 \neq 0 \Bigr\} \\
    \cup \Bigl\{
      \beta_0 = 0, \alpha_0 \neq 0, \;
      \beta_1 = 0, \alpha_1 \neq 0 \Bigr\}.
  \end{multline*}
\end{lemma}

\begin{proof}
  From the definition of $\CL$ it is easy to see that the operator is
  self-adjoint.

  We have to show that $\CL$ is negative if and only if $(\alpha_0,
  \beta_0, \alpha_1, \beta_1)\in A$.  Without loss of generality we
  can assume $\beta_i \geq 0$ for $i=1,2$ and $\alpha_i \geq 0$ whenever
  $\beta_i = 0$ (since we can replace $(\alpha_i, \beta_i)$ by
  $(-\alpha_i, -\beta_i)$ if required).  Assume that $\lambda$ is an
  eigenvalue of $\CL$.  If $\lambda>0$, the corresponding
  eigenfunctions are of the form
  \begin{equation*}
    u(x) = c_1 \e^{\sqrt{\lambda}x} + c_2 \e^{-\sqrt{\lambda}x}
  \end{equation*}
  where $c_1$ and $c_2$ are given by the boundary conditions: For $u$
  to be in the domain of $\CL$, the coefficients $c_1$ and $c_2$ need
  to satisfy
  \begin{equation*}
    \begin{pmatrix}
      \alpha_0-\beta_0\sqrt{\lambda} & \alpha_0+\beta_0\sqrt{\lambda} \\
      (\alpha_1+\beta_1\sqrt{\lambda}) \e^{\sqrt{\lambda}} & (\alpha_1-\beta_1\sqrt{\lambda}) \e^{-\sqrt{\lambda}}
    \end{pmatrix}
    \begin{pmatrix}
      c_1 \\ c_2
    \end{pmatrix}
    =
    \begin{pmatrix}
      0 \\ 0
    \end{pmatrix}.
  \end{equation*}
  Non-trivial solutions exist only if the matrix is singular or,
  equivalently, if its determinant
  \begin{equation}\label{E:ev-exists}
    f(\lambda)
    = \alpha_0\alpha_1 + \bigl(\alpha_0\beta_1+\alpha_1\beta_0 \bigr)\,\sqrt{\lambda}\coth\sqrt\lambda + \beta_0\beta_1\lambda
  \end{equation}
  satisfies $f(\lambda) = 0$.  For $\lambda = 0$ the eigenfunctions
  are of the form $u(x) = c_1 1 + c_2 x$ and an argument similar to
  the one above shows that the boundary conditions can be satisfied if
  and only if $\alpha_0\alpha_1 + \alpha_0\beta_1 +\alpha_1\beta_0 =
  0$.  Since $x \coth(x) \to 1$ as $x\to 0$, this condition can be
  written as $f(0) = 0$ where
  \begin{equation*}
    f(0)
    = \lim_{\lambda\downarrow 0} f(\lambda)
    = \alpha_0\alpha_1 + \alpha_0\beta_1+\alpha_1\beta_0.
  \end{equation*}
  This shows that $\CL$ is negative whenever $f(\lambda) \neq 0$ for
  all $\lambda \geq 0$.

  Let $(\alpha_0, \beta_0, \alpha_1, \beta_1)\in A$.  Assume first
  $\beta_i \neq 0$ for $i=1,2$ and let $\xi_i = \alpha_i/\beta_i$.
  Then, by the first condition in $A$,
   we have $\xi_0, \xi_1 > -1$ and $(\xi_0 - 1)(\xi_1 - 1) > 1$, and
   for $\lambda \geq 0$ we get
  \begin{equation*}
    f(\lambda)
    = \xi_0\xi_1 + \bigl(\xi_0+\xi_1 \bigr)\,\sqrt{\lambda}\coth\sqrt\lambda + \lambda
    \geq (\xi_0 + 1) (\xi_1 + 1) - 1
    > 0.
  \end{equation*}
  The cases $\beta_0 = 0$ or $\beta_1 = 0$
  can be treated similarly.  Thus, for $(\alpha_0, \beta_0, \alpha_1,
  \beta_1)\in A$, there are no eigenvalues with $\lambda \geq 0$ and
  the operator is negative.

  For the converse statement, assume that $(\alpha_0, \beta_0,
  \alpha_1, \beta_1)\notin A$.  We then have to show that there is a
  $\lambda\geq0$ with $f(\lambda) = 0$.  Assume first $\beta_i > 0$
  for $i=1,2$ and define $\xi_i$ as above.  If $(\xi_0+1)(\xi_1+1) =
  1$, we have $f(0) = 0$.  If $(\xi_0+1)(\xi_1+1) < 1$, the function
  $f$ satisfies $f(0) < 0$ and $f(\lambda) \to\infty$ as
  $\lambda\to\infty$; by continuity there is a $\lambda>0$ with
  $f(\lambda) = 0$.  Finally, if $(\xi_0+1)(\xi_1+1) > 1$ but $\xi_0,
  \xi_1 \leq -1$, we have $f(0) > 0$ and for $\lambda$ with
  $\sqrt{\lambda} = -(\xi_0+\xi_1)/2 > 0$ we find $f(\lambda) <
  \xi_0\xi_1 + \bigl(\xi_0+\xi_1\bigr)\,\sqrt\lambda + \lambda =
  \xi_0\xi_1 - (\xi_0+\xi_1)^2/4 = -(\xi_0-\xi_1)^2/4 \leq 0$ and by
  continuity there is a $\lambda>0$ with $f(\lambda) = 0$.  Again, the
  cases $\beta_0 = 0$ and $\beta_1 = 0$ can be treated similarly.
\end{proof}

A representation of the eigenvalues of $\CL$ which is similar to the
one in the proof of lemma~\ref{L:cond-neg} can be found in section~3
of~\citet{CaccFin10}.

The statement from lemma~\ref{L:cond-neg} reproduces the well-known
results that the Laplacian with Dirichlet boundary conditions
($\alpha_i = 1, \beta_i = 0$) is negative definite whereas the
Laplacian with von~Neumann boundary conditions ($\alpha_i = 0, \beta_i
= 1$) is not (since constants are eigenfunctions with eigenvalue~$0$).

\begin{lemma}\label{L:C-exact}
  Let $\CL$ be negative definite.  Then the following statements hold:
  \begin{jvlist}
  \item[1.] The linear SPDE~\eqref{E:linear} has global, continuous
    $\CH$-valued solutions.
  \item[2.] Equation~\eqref{E:linear} has a unique stationary
    distribution $\nu$ on $\CH$.  The measure~$\nu$ is Gaussian with
    mean~$0$ and covariance function
    \begin{equation}\label{E:C-exact}
      C(x,y)
      = \frac{\beta_0\beta_1 + \alpha_0 \beta_1 \, xy + \beta_0\alpha_1(1-x)(1-y)}%
               {\alpha_0\alpha_1 + \alpha_0\beta_1 + \beta_0\alpha_1} + x\wedge y - xy
    \end{equation}
    where $x \wedge y$ denotes the minimum of $x$ and~$y$.
  \item[3.] The measure $\nu$ coincides with the distribution of
    $U\in \rC\bigl([0,1],\R\bigr)$ given by
    \begin{equation*}
      U(x) = (1-x)L + xR + B(x) \qquad \forall x\in[0,1],
    \end{equation*}
    where $L\sim\CN(0, \SL^2)$, $R\sim\CN(0,\SR^2)$ with $\Cov(L,R) =
    \SLR$, the process $B$ is a Brownian bridge, independent of $L$
    and~$R$, and
    \begin{multline*}
      \SL^2 = \frac{\beta_0(\alpha_1+\beta_1)}{\alpha_0\alpha_1 + \alpha_0\beta_1 + \beta_0\alpha_1}, \quad
      \SR^2 = \frac{(\alpha_0+\beta_0)\beta_1}{\alpha_0\alpha_1 + \alpha_0\beta_1 + \beta_0\alpha_1}, \\
      \SLR = \frac{\beta_0\beta_1}{\alpha_0\alpha_1 + \alpha_0\beta_1 + \beta_0\alpha_1}.
    \end{multline*}
  \end{jvlist}
\end{lemma}

\begin{proof}
  From \citet{IscMaMcDoTaZi90} and \citet{DaPraZa96}
  \citep[see][Lemma~2.2]{HaiStuaVoWi05} we know that \eqref{E:linear}
  has global, continuous $\CH$-valued solutions as well as a unique
  stationary distribution given by $\nu = \CN(0, -\CL^{-1})$.  An easy
  computation shows that $C$ as given in equation~\eqref{E:C-exact} is
  a Green's function for the operator $-\CL$, \textit{i.e.}\ $-\d_x^2
  C(x,y) = \delta(x-y)$ and for every $y\in(0,1)$ the function $x
  \mapsto C(x,y)$ satisfies the boundary conditions~\eqref{E:domain}.
  This completes the proof of the first two statements.

  For the third statement we note that $U$ is centred Gaussian with
  covariance function
  \begin{align*}
    C(x,y)
      &= \Cov\bigl(U(x),U(y) \bigr) \\
      &= \Cov\bigl((1-x)L + xR + B(x), (1-y)L + yR + B(y)\bigr) \\
      &= (1-x)(1-y) \,\SL^2 + \bigl((1-x)y+x(1-y)\bigr) \,\SLR + xy \,\SR^2 \\
      &\hskip1cm
          + x\wedge y - xy.
  \end{align*}
  The fact that this covariance function can be written in the
  form~\eqref{E:C-exact} can be checked by a direct calculation.
\end{proof}

Using the results for the linear SPDE~\eqref{E:linear} we can
now study the full SPDE~\eqref{E:SPDE}.  The result is given in the
following lemma.

\begin{lemma}\label{L:density}
  Let $\CL$ be negative definite.  Furthermore, let $f = F'$ where
  $F\in \rC^2(\R, \R)$ is bounded from above with bounded second
  derivative.  Then the following statements hold:
  \begin{jvlist}
  \item[1.] The nonlinear SPDE~\eqref{E:SPDE} has global, continuous
    $\CH$-valued solutions.
  \item[2.] Equation~\eqref{E:SPDE} has a unique stationary
    distribution~$\mu$ which is given by
  \begin{equation*}
    \frac{d\mu}{d\nu}(u)
    = \frac{1}{Z} \exp\Bigl(
        \int_0^1 F\bigl(u(x)\bigr) \,dx
      \Bigr)
  \end{equation*}
  where $\nu$ is the stationary distribution of~\eqref{E:linear} from
  lemma~\ref{L:C-exact} and $Z$ is the normalisation constant.
\end{jvlist}
\end{lemma}

\begin{proof}
  This result is well known, see \textit{e.g.}\ \citet{Za89} or
  \citet[corollary~4.5]{HaiStuaVo07}.
\end{proof}

\section{Finite Element Approximation}
\label{S:SDE}

In this section we consider finite dimensional approximations of the
SPDE~\eqref{E:SPDE}, obtained by discretising space using the finite
element method.  The approximation follows the
same approach as for deterministic PDEs.  For background on the
deterministic case we refer to \cite{BreSco02} or~\cite{Jo90}.

\smallskip

To discretise space, let $n\in\N$, $\DX = 1/n$ and consider $x$-values
on the grid $k\DX$ for $k\in\N$.  Since the differential operator
$\CL$ in~\eqref{E:SPDE} is a second order differential operator, we
can choose a finite element basis consisting of ``hat functions''
$\phi_i$ for $i\in\Z$ which have $\phi_i(i \,\DX)=1$, $\phi_i(j \,\DX)
= 0$ for all $j\neq i$, and which are affine between the grid points.
Formally, the weak (in the PDE-sense)
formulation of SPDE~\eqref{E:SPDE} can be written as
\begin{equation*}
   \la v, du(t) \ra = B(v, u) \,dt + \la v, f\bigl(u(t)\bigr) \ra + \sqrt{2} \la v, dw(t)\ra
\end{equation*}
where $\la \,\cdot\,, \,\cdot\,\ra$ denotes the $\rL^2$-inner product
and the bilinear form~$B$ is given by
\begin{equation*}
  B(u, v)
  = \la v, \CL u \ra
  = u(1) v'(1) - u(0) v'(0) - \int_0^1 u'(x) v'(x) \,dx.
\end{equation*}
The discretised solution is found by taking $u$ and~$v$ to be in the
space spanned by the functions~$\phi_i$, \textit{i.e.}\ by using the ansatz
\begin{equation*}
  u(t) = \sum_j U_j(t) \phi_j
\end{equation*}
and then considering the following system of equations:
\begin{equation}\label{E:FE-SDE0}
  \la \phi_i, \sum_j dU_j \phi_j\ra
  = \la \phi_i, \d_x^2 \sum_j U_j \phi_j\ra \,dt
    + \la \phi_i, f\bigl(\sum_j U_j \phi_j\bigr) \ra \,dt
    + \sqrt2 \la \phi_i, dw\ra.
\end{equation}

The domain $V$ of the bilinear form $B$ depends on the boundary
conditions of~$\CL$; there are four different cases:
\begin{enumerate}
\item If $\beta_0, \beta_1 \neq 0$ in~\eqref{E:domain}, \textit{i.e.}\
  for von~Neumann or Robin boundary conditions, we have $V =
  \rH^1\bigl([0,1], \R\bigr)$ and we consider the basis functions
  $\phi_i$ for $i\in I = \{ 0, 1, \ldots, n-1, n \}$.
\item If $\beta_0 = 0$ and $\beta_1 \neq 0$, \textit{i.e.}\ for a
  Dirichlet boundary condition at the left boundary, we have $V =
  \bigl\{ u \in \rH^1\bigl([0,1], \R\bigr) \bigm| u(0)=0 \bigr\}$ and we
  consider the basis functions $\phi_i$ for $i\in I = \{ 1, \ldots,
  n-1, n \}$.
\item If $\beta_0 \neq 0$ and $\beta_1 = 0$, \textit{i.e.}\ for a
  Dirichlet boundary condition at the right boundary, we have $V =
  \bigl\{ u \in \rH^1\bigl([0,1], \R\bigr) \bigm| u(1)=0 \bigr\}$ and we
  consider the basis functions $\phi_i$ for $i\in I = \{ 0, 1, \ldots,
  n-1 \}$.
\item If $\beta_0, \beta_1 = 0$, \textit{i.e.}\ for Dirichlet boundary
  conditions at both boundaries, we have $V = \bigl\{ u \in
  \rH^1\bigl([0,1], \R\bigr) \bigm| u(0) = u(1) = 0 \bigr\}$ and we
  consider the basis functions $\phi_i$ for $i\in I = \{ 1, \ldots,
  n-1 \}$.
\end{enumerate}
Throughout the rest of the text we will write $I$ for the index set of
the finite element discretisation as above, and the discretised
solution $u = \sum_{j\in I} U_j \phi_j$ will be described by the
coefficient vector $U\in\R^I$.  In all cases we define the ``stiffness
matrix'' $L^{(n)} \in \R^{I\times I}$ by $L^{(n)}_{ij} = B(\phi_i,
\phi_j)$ for all $i,j \in I$.  For the given basis functions we get
\begin{equation*}
  L^{(n)}_{ij} =
  \begin{cases}
    - \frac{2}{\DX} & \mbox{if $i=j \notin\{0, n\}$,} \\
    + \frac{1}{\DX} & \mbox{if $i\in\{j-1, j+1\}$,} \\
    - \frac{1}{\DX} - \frac{\alpha_0}{\beta_0}& \mbox{if $i=j=0$,} \\
    - \frac{1}{\DX} - \frac{\alpha_1}{\beta_1}& \mbox{if $i=j=n$,} \\
    0 & \mbox{else,}
  \end{cases}
\end{equation*}
where the cases $i=j=0$ and $i=j=n$ cannot occur for Dirichlet
boundary conditions.  The ``mass matrix'' $M\in\R^{I\times I}$ is
defined by $M_{ij} = \la \phi_i, \phi_j\ra$ and for $i,j \in I$ we get
\begin{equation*}
  M_{ij} =
  \begin{cases}
    \frac{4}{6} \DX & \mbox{if $i=j \notin\{0, n\}$,} \\
    \frac{1}{6} \DX & \mbox{if $i\in\{j-1, j+1\}$,} \\
    \frac{2}{6} \DX & \mbox{if $i=j \in\{0, n\}$,} \\
    0 & \mbox{else,}
  \end{cases}
\end{equation*}
where, again, the cases $i=j=0$ and $i=j=n$ don't occur for Dirichlet
boundary conditions.  We note that the matrix~$L^{(n)}$ only has
the prefactor $1/\DX$ instead of the $1/\DX^2$ one would expect for a
second derivative.  The ``missing'' $\DX$ appears in the matrix~$M$.

Since
\begin{equation*}
  \Cov(\la\phi_i, w\ra, \la\phi_j, w\ra) = \la\phi_i, \phi_j\ra = M_{ij},
\end{equation*}
equation~\eqref{E:FE-SDE0} can be written as
\begin{equation*}
  M \, dU_t
  = L^{(n)} U_t \,dt + f_n(U_t) \,dt + \sqrt{2} M^{1/2} dW_t
\end{equation*}
where $f_n\colon \R^I\to \R^I$ is defined by
\begin{equation}\label{E:FE-f}
  f_n(u)_i = \bigl\la \phi_i, f\bigl(\sum_{j\in I} u_j \phi_j\bigr) \bigr\ra
\end{equation}
for all $u\in\R^I$ and $i\in I$.  Multiplication with $M^{-1}$ then
yields the following SDE describing the evolution of the
coefficients~$(U_i)_{i\in I}$:

\begin{definition}
  The {\em finite element discretisation} of SPDE~\eqref{E:SPDE} is given by
  \begin{equation}\label{E:FE-SDE}
    dU_t
    = M^{-1} L^{(n)} U_t \,dt
      + M^{-1} f_n(U_t) \,dt
      + \sqrt{2} M^{-1/2}\, dW_t
  \end{equation}
  where $W$ is an $|I|$-dimensional standard Brownian motion, $I\subseteq \{0, 1, \ldots, n\}$
  is the index set of the finite element discretisation, and
  $L^{(n)}$ and $M$ are as above.
\end{definition}

Our aim is to show that the stationary distribution of~\eqref{E:FE-SDE}
converges to the stationary distribution of the SPDE~\eqref{E:SPDE}.
We start our analysis by considering the linear
case $f\equiv 0$.  For this case the finite element discretisation
simplifies to~\eqref{E:FE-linear} below.

\begin{lemma}\label{L:FE-linear}
  Let $\CL$ be negative definite.  Then $\nu_n = \CN\bigl(0,
  (-L^{(n)})^{-1}\bigr)$ is the unique stationary distribution of
  \begin{equation}\label{E:FE-linear}
    dU_t
    = M^{-1} L^{(n)} U_t \,dt
      + \sqrt{2} M^{-1/2}\, dW_t.
  \end{equation}
\end{lemma}

\begin{proof}
  Since $\CL$ is a negative operator, the matrix $L^{(n)}$ is a
  symmetric, negative definite matrix.  As the product of a positive
  definite symmetric matrix and a negative definite symmetric matrix,
  $M^{-1}L^{(n)}$ is negative definite; its eigenvalues coincide with
  the eigenvalues of
  \begin{equation*}
    M^{-1/2} L^{(n)} M^{-1/2}
    = - \bigl((-L^{(n)})^{1/2}M^{-1/2}\bigr)^\top \, \bigl((-L^{(n)})^{1/2} M^{-1/2}\bigr).
  \end{equation*}
  From \citet[theorem~8.2.12]{Arn74} we know that
  then the unique stationary distribution of the
  SDE~\eqref{E:FE-linear} is~$\CN(0,C^{(n)})$ where $C^{(n)}$ solves the
  Lyapunov equation
  \begin{equation*}
    M^{-1}L^{(n)} C^{(n)} + C^{(n)} L^{(n)} M^{-1} = - 2 M^{-1}.
  \end{equation*}
  By theorem~5.2.2 of~\citet{Lancaster-Rodman95}, this system of
  linear equations has a unique solution and it is easily verified
  that this solution is given by $C^{(n)} = (-L^{(n)})^{-1}$.
\end{proof}

The following lemma shows that for $f\equiv 0$ there is no
discretisation error at all: the stationary distributions of the SPDE
\eqref{E:linear}, projected to~$\R^I$, and of the finite element
discretisation~\eqref{E:FE-SDE} coincide.

\begin{lemma}\label{L:exact}
  Define $\Pi\colon \rC\bigl([0,1],\R\bigr)\to \R^I$ by
  \begin{equation}\label{E:Pi}
    (\Pi u)_i = u (i\DX) \qquad \forall i\in I.
  \end{equation}
  Let $\nu$ be the stationary distribution of the linear
  SPDE~\eqref{E:linear} on $\rC\bigl([0,1],\R\bigr)$ and let $\nu_n$ be
  the stationary distribution of the linear finite element
  discretisation~\eqref{E:FE-linear}.  Then we have
  \begin{equation*}
    \nu_n = \nu\circ\Pi^{-1}
  \end{equation*}
  for every $n\in\N$.
\end{lemma}

\begin{proof}
  Let $C\exact$ be the covariance matrix of $\nu\circ\Pi^{-1}$ and let
  $C^{(n)}$ be the covariance matrix of~$\nu_n$.  Since both measures
  under consideration are centred Gaussian, if suffices to show
  $C\exact = C^{(n)}$.
  By lemma~\ref{L:C-exact}, the matrix
  $C\exact$ satisfies
  \begin{equation*}
    C\exact_{i,j} = C(i\DX, j\DX) \qquad \forall i,j \in I
  \end{equation*}
  where $C$ is given by equation~\eqref{E:C-exact}.  By
  lemma~\ref{L:FE-linear} we have $C^{(n)} = (-L^{(n)})^{-1}$.  A
  simple calculation, using the fact that both $C\exact$ and $L^{(n)}$
  are known explicitly, shows $C\exact L^{(n)} = -I$ and thus
  $C\exact=C^{(n)}$ (the four different cases for the boundary
  conditions need to be checked separately).  This completes the
  proof.
\end{proof}

The preceding results only consider the linear case and for the
general case, in the presence of the non-linearity~$f$, we can of
course no longer expect a similar result to hold.  As a starting point
for analysing this case, we reproduce a well-known result which allows
to identify the stationary distribution of the discretised finite
element equation.

\begin{lemma}\label{L:preconditioning}
  Let $F\in \rC^2\bigl(\R^d, \R)$ with bounded second derivatives and
  satisfying the condition $Z = \int_{\R^d} \e^{2 F(x)} \,dx <
  \infty$.  Furthermore, let $A\in\R^{d\times d}$ be invertible.  Then
  the SDE
  \begin{equation}\label{E:grad-log-phi-SDE}
    dX_t = A A^\top \nabla F(X_t) \,dt + A \,dW_t
  \end{equation}
  has a unique stationary distribution which has density
  \begin{equation*}
    \phi(x) = \frac1Z \e^{2F(x)}
  \end{equation*}
  with respect to the Lebesgue measure on~$\R^d$.
\end{lemma}

\begin{proof}
  Define $G(y) = F(Ay)$ for all $y\in \R^d$.  By the assumptions on
  $F$ we have $G\in \rC^2\bigl(\R^d, \R)$ with bounded second
  derivatives and $Z_G = \int_{\R^d} \e^{2 G(y)} \,dy < \infty$.
  Therefore, the SDE
  \begin{equation*}
    dY_t = \nabla G(Y_t) \,dt + dW_t
  \end{equation*}
  has a unique stationary distribution with density
  \begin{equation*}
    \psi(y) = \frac{1}{Z_G} \e^{2G(y)}.
  \end{equation*}
  Since $\nabla G(y) = A^\top \nabla F(A y)$, we have
  \begin{equation*}
    dY_t = A^\top \nabla F(A Y_t) \,dt + dW_t
  \end{equation*}
  and multiplying this equation by $A$ gives
  \begin{equation*}
    d(AY_t) = A A^\top \nabla F(A Y_t) \,dt + A\,dW_t.
  \end{equation*}

  Consequently, $X_t = AY_t$ satisfies the
  SDE~\eqref{E:grad-log-phi-SDE} and has a unique stationary
  distribution with density proportional to $\psi(A^{-1}x) \propto
  \e^{2G(A^{-1}x)} = \e^{2F(x)}$.  Since this function, up to a
  multiplicative constant, coincides with $\phi$, the process $X$ has
  stationary density~$\phi$.
\end{proof}

Because the stationary distribution in the lemma does not depend
on~$A$, the stationary distribution of~\eqref{E:grad-log-phi-SDE} does
not change when we remove/add $A$ from the equation.  The process of
introducing the matrix $A$ is sometimes called ``preconditioning the
SDE''.

In cases were we are only interested in the stationary distribution of
a discretised SPDE, the argument from lemma~\ref{L:preconditioning}
allows us to omit the mass matrix~$M$ from the finite element
SDE~\eqref{E:FE-SDE}.  In particular we don't need to consider the
potentially computationally expensive square root~$M^{1/2}$ in
numerical simulations.

\begin{lemma}\label{L:FE-density}
  Let $\CL$ be negative definite.  Furthermore, let $f = F'$ where
  $F\in \rC^2(\R, \R)$ is bounded from above with bounded second
  derivative.  Then the finite element SDE~\eqref{E:FE-SDE} has a
  unique stationary distribution~$\mu_n$ given by
  \begin{equation}\label{E:mu-n-dens}
    \frac{d\mu_n}{d\nu_n} = \frac{1}{Z_n} \exp\bigl( F_n \bigr)
  \end{equation}
  where
  \begin{equation*}
    F_n(u) = \int_0^1 F\Bigl(\sum_{j\in I} u_j \phi_j(t)\Bigr) \,dt
      \qquad \forall u\in\R^I,
  \end{equation*}
  $Z_n$ is the normalisation constant and $\nu_n$ is the stationary
  distribution of the linear equation from lemma~\ref{L:FE-linear}.
\end{lemma}

\begin{proof}
  Let $\Phi(u) = \frac12 u^\top L^{(n)} u + F_n(u)$ for all $u\in\R^I$.
  Then
  \begin{equation*}
    \d_i \Phi(u)
    = (L^{(n)}u)_i + \bigl\la \phi_i, F'\bigl(\sum_{j\in I} u_j \phi_j(t) \bigr) \bigr\ra
    = \bigl( L^{(n)}u + f_n(u) \bigr)_i
  \end{equation*}
  for all $i\in I$ and thus \eqref{E:FE-SDE} can be written as
  \begin{equation*}
    dU_t = M^{-1} \nabla\Phi(U_t) \,dt + \sqrt{2} M^{-1/2} \, dW_t.
  \end{equation*}
  By lemma~\ref{L:preconditioning}, this SDE has a unique stationary
  distribution~$\mu_n$ whose density w.r.t.\ the $|I|$-dimensional
  Lebesgue measure $\lambda$ is given by
  \begin{equation*}
    \frac{d\mu_n}{d\lambda}(u)
    = \frac{1}{\tilde Z_n} \e^{\Phi(u)}
    = \frac{1}{\tilde Z_n} \exp\Bigl(- \frac12 u^\top (-L^{(n)}) u + F_n(u) \Bigr).
  \end{equation*}

  From lemma~\ref{L:FE-linear} we know that the density of $\nu_n$
  w.r.t.~$\lambda$ is
  \begin{equation*}
    \frac{d\nu_n}{d\lambda}(x)
    = \frac{1}{\bigl(2\pi\bigr)^{|I| / 2} \bigl(\det(-L^{(n)})\bigr)^\frac12}
        \exp\Bigl(-\frac12 x^{\mathrm{T}} (-L^{(n)}) x\Bigr)
  \end{equation*}
  and consequently the distribution $\mu_n$ satisfies
  \begin{equation*}
    \frac{d\mu_n}{d\nu_n}(x)
    = \frac{d\mu_n}{d\lambda}(x) / \frac{d\nu_n}{d\lambda}(x)
    \propto \exp\bigl( F_n \bigr).
  \end{equation*}
  Since the right-hand side, up to constants, coincides with the
  expression in~\eqref{E:mu-n-dens}, the proof is complete.
\end{proof}

\section{Main Result}
\label{S:result}

Now we have identified the stationary distribution of the SPDE
(in section~\ref{S:SPDE}) and of the SDE (in section~\ref{S:SDE}),
we can compare the two stationary distributions.
The result is given in the following theorem.

\begin{theorem}\label{T:result}
  Let $\mu$ be the stationary distribution of the SPDE~\eqref{E:SPDE}
  on $\rC\bigl([0,1],\R\bigr)$.  Let $\mu_n$ be the stationary
  distribution of the finite element equation~\eqref{E:FE-SDE}
  on~$\R^I$.  Let $\CL$ be negative and assume $f=F'$ where $F\in
  \rC^2(\R)$ is bounded from above with bounded second derivative.
  Then
  \begin{equation*}
    \bigl\| \mu\circ\Pi^{-1} - \mu_n \bigr\|_{\mathrm{TV}} = \CO\bigl(\frac1n\bigr)
  \end{equation*}
  as $n\to\infty$, where $\| \quark \|_{\mathrm{TV}}$ denotes
  total-variation distance between probability distributions on~$\R^I$.
\end{theorem}

Before we prove this theorem, we first show some auxiliary results.
The following lemma will be used to get rid of the (not explicitly
known) normalisation constant~$Z_n$.

\begin{lemma}\label{L:normalising-constants}
  Let $(\Om, \CF, \mu)$ be a measure space and $f_1, f_2\colon \Om\to
  [0,\infty]$ integrable with $Z_i = \int f_i \,d\mu > 0$ for $i=1,2$.
  Then
  \begin{equation*}
    \int \Bigl| \frac{f_1}{Z_1} - \frac{f_2}{Z_2} \Bigr| \,d\mu
    \leq \frac{2}{\max(Z_1, Z_2)} \int \bigl| f_1 - f_2 \bigr| \,d\mu.
  \end{equation*}
\end{lemma}

\begin{proof}
  Using the $\rL^1$-norm $\|f\| = \int |f| \,d\mu$ we can write
  \begin{equation*}
    \bigl\| \frac{f_1}{Z_1} - \frac{f_2}{Z_2} \Bigr\|
    \leq \Bigl\| \frac{f_1}{Z_1} - \frac{f_2}{Z_1} \Bigr\|
      + \Bigl\| \frac{f_2}{Z_1} - \frac{f_2}{Z_2} \Bigr\| \\
    = \frac{1}{Z_1} \bigl\| f_1 - f_2 \bigr\|
      + \frac{|Z_2 - Z_1|}{Z_1Z_2} \bigl\|f_2\bigr\|.
  \end{equation*}
  Since $Z_i = \|f_i\|$ we can conclude
  \begin{equation*}
    \bigl\| \frac{f_1}{Z_1} - \frac{f_2}{Z_2} \Bigr\|
    \leq \frac{1}{Z_1} \bigl\| f_1 - f_2 \bigr\|
      + \frac{\bigl|\|f_2\| - \|f_1\|\bigr|}{Z_1}
    \leq \frac{2}{Z_1} \bigl\| f_1 - f_2 \bigr\|
  \end{equation*}
  where the second inequality comes from the inverse triangle
  inequality.  Without loss of generality we can assume $Z_1 \geq
  Z_2$ (otherwise interchange $f_1$ and $f_2$ in the above argument)
  and thus the claim follows.
\end{proof}

\begin{lemma}\label{L:Pi}
  Let $\mu$ and $\nu$
  be probability measures on $\rC\bigl([0,1],\R\bigr)$ with $\mu\ll\nu$
  and let $\Pi\colon \rC\bigl([0,1],\R\bigr)\to\R^I$
  be the projection from~\eqref{E:Pi}.
  Then $\mu\circ\Pi^{-1} \ll \nu\circ\Pi^{-1}$ and
  \begin{equation*}
    \frac{d(\mu\circ\Pi^{-1})}{d(\nu\circ\Pi^{-1})}\circ\Pi
    = \E_\nu\bigl( \frac{d\mu}{d\nu} \bigm| \Pi\bigr).
  \end{equation*}
\end{lemma}

\begin{proof}
  Let $\phi = \frac{d\mu}{d\nu}$.  Since $\E(\phi|\Pi)$ is
  $\Pi$-measurable, there is a function $\psi\colon \R\to \R$ with
  $\E(\phi|\Pi) = \psi\circ\Pi$.  Let $A\subseteq\R$ be measurable.
  Then
  \begin{multline*}
    \int_{\R} \psi \, 1_A \,d(\nu\circ\Pi^{-1})
    = \int_{\rC\bigl([0,1],\R\bigr)} \psi\circ\Pi \, 1_{\Pi^{-1}(A)} \,d\nu
    = \int_{\rC\bigl([0,1],\R\bigr)} \E(\phi|\Pi) \, 1_{\Pi^{-1}(A)} \,d\nu \\
    = \int_{\rC\bigl([0,1],\R\bigr)} \phi \, 1_{\Pi^{-1}(A)} \,d\nu
    = \mu\circ\Pi^{-1}(A)
  \end{multline*}
  by the definition of conditional expectation.  This shows that
  $\psi$ is indeed the required density.
\end{proof}

\begin{proof} (of theorem~\ref{T:result}).
  Let $\Pi$, $\nu$ and $\nu_n$ as in lemma~\ref{L:exact}.  Using
  lemmata \ref{L:exact} and~\ref{L:Pi} we find
  \begin{equation}\label{E:proof-start}
    \begin{split}
    \bigl\| \mu\circ\Pi^{-1} - \mu_n \bigr\|_{\mathrm{TV}}
    &= \E_{\nu_n}\Bigl| \frac{d\mu\circ\Pi^{-1}}{d\nu_n} - \frac{d\mu_n}{d\nu_n} \Bigr| \\
    &= \E_\nu \Bigl| \frac{d\mu\circ\Pi^{-1}}{d\nu\circ\Pi^{-1}}\circ\Pi - \frac{d\mu_n}{d\nu_n}\circ\Pi \Bigr| \\
    &= \E_\nu\Bigl| \E_\nu\bigl( \frac{d\mu}{d\nu} - \frac{d\mu_n}{d\nu_n}\circ\Pi \bigm| \Pi\bigr) \Bigr| \\
    &\leq \E_\nu\Bigl| \frac{d\mu}{d\nu} - \frac{d\mu_n}{d\nu_n}\circ\Pi \Bigr|.
  \end{split}
  \end{equation}
  From lemma~\ref{L:density} we know
  \begin{equation*}
    \frac{d\mu}{d\nu}(U)
    = \frac{1}{Z} \exp\Bigl( \int_0^1 F\bigl(U_x\bigr) \,dx \Bigr).
  \end{equation*}
  Lemma~\ref{L:FE-density} gives
  \begin{equation*}
    \frac{d\mu_n}{d\nu_n}
    = \frac{1}{Z_n} \exp\bigl( F_n \bigr),
  \end{equation*}
  and by the definition of $F_n$ we have
  \begin{equation*}
    \frac{d\mu_n}{d\nu_n}\circ\Pi(U)
    = \frac{1}{Z_n} \exp\Bigl(\int_0^1 F\bigl(U^{(n)}_x\bigr) \,dx \Bigr)
  \end{equation*}
  where $U_u = \sum \Pi(U)_j \phi_j$ for all $U\in
  \rC\bigl([0,1],\R\bigr)$ and $\phi_j$, $j\in I$ are the finite
  element basis functions.  Using lemma~\ref{L:normalising-constants} we
  get
  \begin{multline*}
    \bigl\| \mu\circ\Pi^{-1} - \mu_n \bigr\|_{\mathrm{TV}}
    \leq \E_\nu\Bigl| \frac{d\mu}{d\nu}
                        - \frac{d\mu_n}{d\nu_n}\circ\Pi \Bigr| \\
    \leq \frac{2}{Z} \int \Bigl|
 \exp\Bigl( \int_0^1 F\bigl(U_x\bigr) \,dx \Bigr)
 - \exp\Bigl(\int_0^1 F\bigl(U^{(n)}_x\bigr) \,dx \Bigr)
 \Bigr| \,d\nu(U).
  \end{multline*}
  Since the inequality $\bigl|\e^x - 1 \bigr| \leq |x|
  \exp\bigl(|x|\bigr)$ holds for all $x\in\R$, we conclude
  \begin{equation*}
    \begin{split}
      &\hskip-5mm \bigl\| \mu\circ\Pi^{-1} - \mu_n \bigr\|_{\mathrm{TV}} \\
      &\leq \frac{2}{Z}
        \E \left( \exp\Bigl( \int_0^1 F\bigl(U^{(n)}_x\bigr) \,dx \Bigr)
        \cdot \Bigl|
            \exp\Bigl( \int_0^1 F\bigl(U_x\bigr) - F\bigl(U^{(n)}_x\bigr) \,dx \Bigr)
            - 1 \Bigr| \right) \\
      &\leq \frac{2}{Z}
        \E \left( \exp\Bigl( \int_0^1 F\bigl(U^{(n)}_x\bigr) \,dx \Bigr) \right. \\
       &\hskip1.5cm
        \left. \cdot
        \Bigl| \int_0^1 F\bigl(U_x\bigr) - F\bigl(U^{(n)}_x\bigr) \,dx \Bigr|
            \cdot \exp\Bigl( \bigl| \int_0^1 F\bigl(U_x\bigr) - F\bigl(U^{(n)}_x\bigr) \,dx \bigr| \Bigr)
            \right)
    \end{split}
  \end{equation*}
  where $U$ is distributed according to the Gaussian measure~$\nu$.
  Since $F$ is bounded from above we can estimate the the first
  exponential in the expectation by a constant.  Using the
  Cauchy-Schwarz inequality we get
  \begin{multline}\label{E:central-estimate}
    \bigl\| \mu\circ\Pi^{-1} - \mu_n \bigr\|_{\mathrm{TV}} \\
    \leq c_1
      \Bigl\| \int_0^1 F\bigl(U_x\bigr) - F\bigl(U^{(n)}_x\bigr) \,dx \Bigr\|_2
      \cdot
      \Bigl\| \exp\Bigl( \bigl| \int_0^1 F\bigl(U_x\bigr) - F\bigl(U^{(n)}_x\bigr) \,dx \bigr| \Bigr) \Bigr\|_2
  \end{multline}
  for some constant~$c_1$.

  \begin{figure}
    \begin{center}
      \begin{tikzpicture}
\draw[thin] (0.00,0.00) -- (0.05,-0.15) -- (0.10,-0.56) -- (0.15,-0.55) -- (0.20,0.01) -- (0.25,-0.20) -- (0.30,0.11) -- (0.35,0.05) -- (0.40,0.22) -- (0.45,-0.19) -- (0.50,-0.34) -- (0.55,-0.24) -- (0.60,-0.04) -- (0.65,-0.01) -- (0.70,0.09) -- (0.75,0.01) -- (0.80,-0.34) -- (0.85,-0.58) -- (0.90,-1.15) -- (0.95,-0.90) -- (1.00,-1.08) -- (1.05,-1.05) -- (1.10,-1.35) -- (1.15,-1.46) -- (1.20,-1.49) -- (1.25,-1.73) -- (1.30,-1.84) -- (1.35,-1.79) -- (1.40,-1.77) -- (1.45,-1.42) -- (1.50,-1.22) -- (1.55,-1.05) -- (1.60,-0.75) -- (1.65,-0.89) -- (1.70,-0.78) -- (1.75,-0.79) -- (1.80,-0.83) -- (1.85,-0.94) -- (1.90,-0.99) -- (1.95,-1.07) -- (2.00,-0.67) -- (2.05,-0.71) -- (2.10,-0.55) -- (2.15,-0.51) -- (2.20,-0.67) -- (2.25,-0.53) -- (2.30,-0.47) -- (2.35,-0.69) -- (2.40,-0.74) -- (2.45,-0.49) -- (2.50,-0.43) -- (2.55,-0.52) -- (2.60,-0.48) -- (2.65,-0.40) -- (2.70,-0.22) -- (2.75,-0.22) -- (2.80,-0.32) -- (2.85,-0.18) -- (2.90,-0.43) -- (2.95,-0.79) -- (3.00,-0.88) -- (3.05,-1.21) -- (3.10,-1.13) -- (3.15,-1.31) -- (3.20,-1.38) -- (3.25,-1.34) -- (3.30,-1.26) -- (3.35,-1.17) -- (3.40,-1.26) -- (3.45,-1.29) -- (3.50,-0.86) -- (3.55,-0.91) -- (3.60,-1.34) -- (3.65,-1.25) -- (3.70,-0.99) -- (3.75,-0.42) -- (3.80,-0.34) -- (3.85,-0.10) -- (3.90,-0.39) -- (3.95,-0.57) -- (4.00,-0.33) -- (4.05,-0.65) -- (4.10,-0.78) -- (4.15,-0.93) -- (4.20,-1.19) -- (4.25,-0.85) -- (4.30,-0.67) -- (4.35,-0.18) -- (4.40,-0.12) -- (4.45,-0.32) -- (4.50,-0.38) -- (4.55,-0.33) -- (4.60,0.12) -- (4.65,-0.03) -- (4.70,-0.27) -- (4.75,-0.27) -- (4.80,0.07) -- (4.85,0.15) -- (4.90,0.33) -- (4.95,0.43) -- (5.00,0.50) -- (5.05,0.89) -- (5.10,0.74) -- (5.15,1.11) -- (5.20,1.39) -- (5.25,1.29) -- (5.30,1.59) -- (5.35,1.70) -- (5.40,1.49) -- (5.45,1.46) -- (5.50,1.39) -- (5.55,1.53) -- (5.60,1.76) -- (5.65,1.55) -- (5.70,1.45) -- (5.75,1.49) -- (5.80,1.46) -- (5.85,1.23) -- (5.90,1.62) -- (5.95,1.44) -- (6.00,1.25) -- (6.05,1.00) -- (6.10,1.33) -- (6.15,1.27) -- (6.20,1.63) -- (6.25,1.55) -- (6.30,1.40) -- (6.35,1.11) -- (6.40,1.42) -- (6.45,1.37) -- (6.50,1.40) -- (6.55,1.22) -- (6.60,0.95) -- (6.65,0.75) -- (6.70,0.51) -- (6.75,0.42) -- (6.80,0.24) -- (6.85,0.32) -- (6.90,0.60) -- (6.95,0.60) -- (7.00,0.73) -- (7.05,0.73) -- (7.10,0.69) -- (7.15,0.60) -- (7.20,0.32) -- (7.25,0.20) -- (7.30,-0.05) -- (7.35,-0.03) -- (7.40,-0.08) -- (7.45,-0.03) -- (7.50,-0.02) -- (7.55,-0.17) -- (7.60,-0.21) -- (7.65,0.00) -- (7.70,-0.23) -- (7.75,-0.18) -- (7.80,-0.39) -- (7.85,0.06) -- (7.90,0.04) -- (7.95,0.14) -- (8.00,0.18) -- (8.05,0.14) -- (8.10,-0.03) -- (8.15,-0.18) -- (8.20,0.18) -- (8.25,-0.34) -- (8.30,0.13) -- (8.35,0.45) -- (8.40,0.46) -- (8.45,0.20) -- (8.50,0.25) -- (8.55,0.16) -- (8.60,-0.00) -- (8.65,0.05) -- (8.70,-0.29) -- (8.75,-0.58) -- (8.80,-0.71) -- (8.85,-0.45) -- (8.90,-0.29) -- (8.95,-0.21) -- (9.00,-0.01) -- (9.05,-0.14) -- (9.10,-0.15) -- (9.15,0.05) -- (9.20,0.26) -- (9.25,0.48) -- (9.30,0.45) -- (9.35,0.07) -- (9.40,0.55) -- (9.45,0.29) -- (9.50,0.27) -- (9.55,0.03) -- (9.60,0.03) -- (9.65,0.11) -- (9.70,0.38) -- (9.75,0.00) -- (9.80,0.19) -- (9.85,0.07) -- (9.90,-0.09) -- (9.95,-0.13) -- (10.00,-0.18);\draw[thick] (0.00,0.00) -- (1.00,-1.08) -- (2.00,-0.67) -- (3.00,-0.88) -- (4.00,-0.33) -- (5.00,0.50) -- (6.00,1.25) -- (7.00,0.73) -- (8.00,0.18) -- (9.00,-0.01) -- (10.00,-0.18);\draw[->]  (-0.3,-2.04) -- (10.5,-2.04) node[below] {$x$};
\draw[thin] (0.00,0.00) -- (0.00,-2.04);
\draw[thin] (1.00,-1.08) -- (1.00,-2.04);
\draw[thin] (2.00,-0.67) -- (2.00,-2.04);
\draw[thin] (3.00,-0.88) -- (3.00,-2.04);
\draw[thin] (4.00,-0.33) -- (4.00,-2.04);
\draw[thin] (5.00,0.50) -- (5.00,-2.04);
\draw[thin] (6.00,1.25) -- (6.00,-2.04);
\draw[thin] (7.00,0.73) -- (7.00,-2.04);
\draw[thin] (8.00,0.18) -- (8.00,-2.04);
\draw[thin] (9.00,-0.01) -- (9.00,-2.04);
\draw[thin] (10.00,-0.18) -- (10.00,-2.04);
\draw (0,-2.04) node[below] {$0$};
\draw (1,-2.04) node[below] {$1/n$};
\draw (2,-2.04) node[below] {$2/n$};
\draw (10,-2.04) node[below] {$1$};
\draw (5.35,1.70) node[above] {$U$};
\draw (5.35,0.76) node[below,xshift=5] {$U^{(n)}$};
\end{tikzpicture}
    \end{center}
    \caption{\label{fig:BB}\it Illustration of the convergence of
      $U^{(n)}$ to $U$.  Under the distribution $\nu$, the path $U$ is
      a Brownian Bridge with random boundary conditions.  Since
      $U^{(n)}$ is the linear interpolation of $U$ between the grid
      points, the difference $U^{(n)}-U$ consists of a chain of $n$
      independent Brownian bridges.}
  \end{figure}
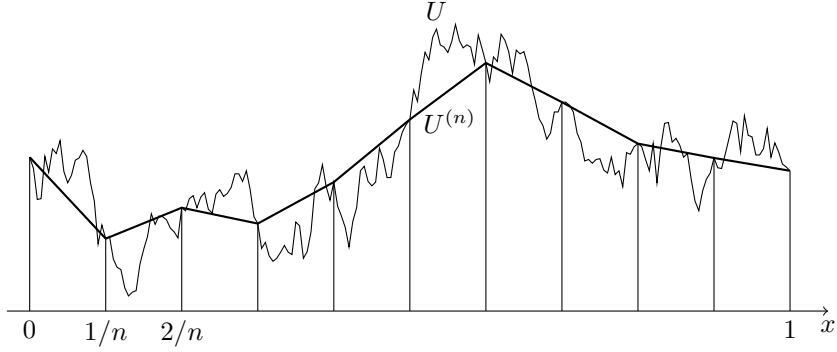

  The main step in the proof is to estimate the right-hand side
  of~\ref{E:central-estimate} by showing that $\bigl| \int_0^1 F(U_x) -
  F(U^{(n)}_x) \,dx \bigr|$ gets small as $n\to\infty$.  By
  lemma~\ref{L:C-exact}, the path $U$ in stationarity is just a
  Brownian bridge (with random boundary values) and, by definition,
  $U^{(n)}$ is the linear interpolation of the values of $U$ at the
  grid points (see figure~\ref{fig:BB} for illustration).  Thus, the
  difference $U^{(n)} - U$ can be written as
  \begin{equation*}
    \bigl(U-U^{(n)}\bigr)(x)
    = \sum_{i=1}^n 1_{[\frac{i-1}{n},\frac{i}{n}]}(x) \; \frac{1}{\sqrt{n}}
        B^{(i)}_{nx - (i-1)}
  \end{equation*}
  where $B^{(1)}, \ldots, B^{(n)}$ are standard Brownian
  bridges, independent of each other and of $U^{(n)}$.
  Using Taylor approximation for $F$ we find
  \begin{multline}\label{E:Taylor}
    \Bigl| \int_0^1 F\bigl(U_x\bigr) - F\bigl(U^{(n)}_x\bigr) \,dx \Bigr| \\
    \leq \Bigl| \int_0^1 f\bigl(U^{(n)}_x\bigr)\bigl(U_x-U^{(n)}_x\bigr) \,dx \Bigr|
         + \frac12 \|F''\|_\infty \int_0^1 \bigl(U_x-U^{(n)}_x\bigr)^2 \,dx \\
    =: \bigl| P_n \bigr| + \frac12 \|F''\|_\infty Q_n.
  \end{multline}

  For the term $|P_n|$ we find
  \begin{equation*}
    P_n
    = \sum_{i=1}^n \int_0^{1/n}
        f\bigl(U^{(n)}_{\frac{i-1}{n}+x}\bigr) \frac{1}{\sqrt{n}} B^{(i)}_{nx} \,dx
    = n^{-3/2} \sum_{i=1}^n \int_0^1 f(U^{(n)}_{\frac{i-1}{n}+y/n}) B^{(i)}_{y} \,dy
  \end{equation*}
  where $B^{(i)}$ are the Brownian bridges defined above.  As an
  abbreviation write $\bar f_i(y) =
  f\bigl(U^{(n)}_{\frac{i-1}{n}+y/n}\bigr)$.  Conditioned on the value
  of $U^{(n)}$, the integrals $\int_0^1 \bar f_i(y) B^{(i)}_{y} \,dy$
  are centred Gaussian with
  \begin{multline*}
  \Var\Bigl(\int_0^1 \bar f_i(y) B^{(i)}_{y} \,dy \Bigm| U^{(n)} \Bigr)
    = \int_0^1 \bar f_i(y) \int_0^1 (y\wedge z - yz) \bar f_i(z) \,dz \,dy \\
    \leq c_2 \| \bar f_i \|_\infty^2
    \leq c_3^2 \bigl( \|U^{(n)}\|_\infty + 1 \bigr)^2
  \end{multline*}
  for some constants $c_2, c_3 > 0$ where the last inequality uses the
  fact the $f$ is Lipschitz continuous.  We get $\E\bigl(P_n\bigm|
  \Pi(U) \bigr) = 0$ and, since the $B^{(i)}$ are independent,
  \begin{equation*}
  \E\bigl( P_n^2 \bigm| U^{(n)} \bigr)
    \leq c_3^2 \bigl( \|U^{(n)}\|_\infty + 1 \bigr)^2  n^{-2}.
  \end{equation*}
  Using the tower property for conditional expectations, and using
  $\|U^{(n)}\|_\infty \leq \|U\|_\infty$, we get
  \begin{equation*}
  \E\bigl( P_n^2 \bigr)
    \leq c_3^2 \E\Bigl(\bigl( \|U\|_\infty + 1 \bigr)^2\Bigr)  n^{-2}.
  \end{equation*}
  Since $U$ is a Gaussian process, $\bigl\| \|U\|_\infty + 1
  \bigr\|_2$ is finite, and we can conclude
  \begin{equation*}
    \E\bigl( P_n^2 \bigr) \leq c_4^2 n^{-2}
  \end{equation*}
  for some constant~$c_4$.  Similarly, for $Q_n$ we find
  \begin{equation}\label{E:Qn}
    Q_n = \sum_{i=1}^n \int_0^1 \bigl(B^{(i)}_y)^2 \,dy \cdot n^{-2}
  \end{equation}
  and thus, using independence of the $B^{(i)}$ again,
  \begin{equation*}
    \E\bigl( Q_n^2 \bigr)
    = c_5^2 n^{-3}
  \end{equation*}
  for some constant~$c_5$.  Combining these estimates we get
  \begin{equation}\label{E:term1}
    \Bigl\| \int_0^1 F\bigl(U_x\bigr) - F\bigl(U^{(n)}_x\bigr) \,dx \Bigr\|_2
    = \bigl\| P_n \bigr\|_2
        + \frac12 \bigl\|F''\bigr\|_\infty \bigl\| Q_n \bigr\|_2
    \leq c_4 n^{-1} + c_5 n^{-3/2}
  \end{equation}
  and thus we have shown the required bound for the first factor
  of~\eqref{E:central-estimate}.

  Finally, we have to show that the second factor
  in~\eqref{E:central-estimate} is bounded, uniformly in~$n$:
  From \eqref{E:Taylor} we get
  \begin{multline}\label{E:term2-start}
    \Bigl\| \exp\Bigl( \bigl| \int_0^1 F\bigl(U_x\bigr) - F\bigl(U^{(n)}_x\bigr) \,dx \bigr| \Bigr) \Bigr\|_2
    \leq \E\Bigl( \exp\bigl( 2 |P_n| + \|F''\|_\infty Q_n \bigr) \Bigr)^{1/2} \\
    \leq \bigl\| \e^{2|P_n|} \bigr\|_2^{1/2}
      \cdot \bigl\| \e^{\|F''\|_\infty Q_n} \bigr\|_2^{1/2}.
  \end{multline}
  It is easy to check that, for all $\sigma>0$, an
  $\CN(0,\sigma^2)$-distributed random variable~$X$ satisfies the
  inequality $\E(\e^{|X|}) \leq 2\e^{\sigma^2/2}$ and thus we have
  \begin{equation}\label{E:exp-P}
    \bigl\| \e^{2|P_n|} \bigr\|_2^{1/2}
      = \E\bigl( \e^{4|P_n|} \bigr)^{1/4}
      \leq 2 \exp(8\, c_4^2 n^{-2})
      < 1
  \end{equation}
  for all sufficiently large~$n$.  Furthermore, using~\eqref{E:Qn}
  and the fact that the $B^{(i)}$ are i.i.d., we find
  \begin{equation}\label{E:exp-Q-start}
    \bigl\| \e^{\|F''\|_\infty Q_n} \bigr\|_2^{1/2}
    = \E\Bigl( \exp\bigl( 2 \|F''\|_\infty |B^{(1)}|_{\rL^2}^2 \, n^{-2} \bigr) \Bigr)^{n/4}
  \end{equation}
  where we write $|\quark|_{\rL^2}$ for the $\rL^2$-norm on the space
  $\rL^2\bigl([0,1],\R\bigr)$.  By Fernique's theorem \citep{Fe70}
  there exists an $\eps>0$ with $\E\bigl( \exp(\eps \,
  |B^{(1)}|_{\rL^2}^2) \bigr) < \infty$.  For all $\lambda>0$ we have
  \begin{equation*}
  \E\bigl(\e^{\lambda |B^{(1)}|_{\rL^2}^2}\bigr)
    = \int_0^\infty \P\bigl(\e^{\lambda |B^{(1)}|_{\rL^2}^2}\geq a\bigr) \,da
    \leq 1 + \int_0^\infty \P\bigl(|B^{(1)}|_{\rL^2}^2 \geq b\bigr) \,\lambda\e^{\lambda b} \,db
  \end{equation*}
  and using Markov's inequality $\P\bigl(|B^{(1)}|_{\rL^2}^2\geq
  b\bigr) \leq \E\bigl(\e^{\eps |B^{(1)}|_{\rL^2}^2}\bigr) \e^{-\eps
    b}$ we get
  \begin{equation*}
  \E\bigl(\e^{\lambda |B^{(1)}|_{\rL^2}^2}\bigr)
    \leq 1 + \int_0^\infty \E\bigl(\e^{\eps |B^{(1)}|_{\rL^2}^2}\bigr) \e^{-\eps b} \,\lambda\e^{\lambda b} \,db
    = 1 + \frac{\lambda}{\eps-\lambda} \E\bigl(\e^{\eps |B^{(1)}|_{\rL^2}^2}\bigr)
  \end{equation*}
  for all $\lambda<\eps$.  Substituting this bound
  into~\eqref{E:exp-Q-start} we have
  \begin{equation}\label{E:exp-Q}
    \bigl\| \e^{\|F''\|_\infty Q_n} \bigr\|_2^{1/2}
    \leq \bigl(1 + \frac{c_6}{n^2} \bigr)^n
    \leq 2
  \end{equation}
  for some constant $c_6$ and all sufficiently large~$n$.  From
  \eqref{E:exp-P} and~\eqref{E:exp-Q} we see that the right-hand
  side of~\eqref{E:term2-start} is bounded uniformly in~$n$,
  \textit{i.e.}
  \begin{equation}\label{E:term2}
    \Bigl\| \exp\Bigl( \bigl| \int_0^1 F\bigl(U_x\bigr) - F\bigl(U^{(n)}_x\bigr) \,dx \bigr| \Bigr) \Bigr\|_2
      \leq c_7
  \end{equation}
  for all $n\in\N$ and some constant~$c_7$.

  Combining \eqref{E:term1} and~\eqref{E:term2} we see that the
  right-hand side in~\eqref{E:central-estimate} is of
  order~$\CO(n^{-1})$.  This completes the proof.
\end{proof}

In the preceding theorem we compared the stationary
distribution~$\mu_n$ of the finite element SDE on $\R^I$ and the
stationary distribution~$\mu$ of the SPDE on $\rC\bigl([0,1],
\R\bigr)$ by projecting $\mu$ onto the finite dimensional space
$\R^I$.  An alternative approach is to embed $\R^I$ into
$\rC\bigl([0,1], \R\bigr)$ instead.  A na{\"\i}ve implementation of
this idea would be to extend vectors from $\R^I$ to continuous
functions via linear interpolation.  Unfortunately, the image of
$\mu_n$ when projected to $\rC\bigl([0,1], \R\bigr)$ in this way would
be mutually singular with $\mu$ and thus the total variation norm
would not provide a useful measure for the distance between the two
distributions.  For this reason, we choose here a different approach,
described in the following definition.

\begin{definition}\label{D:hat}
  Given a probability measure $\mu_n$ on $\R^I$, we define a
  distribution $\hat\mu_n$ as follows: Consider a random variable $X$
  which is distributed according to $\mu_n$.  Given $X$, construct
  $Y\in \rC\bigl([0,1],\R\bigr)$ by setting $Y(k\,\DX) = X_k$ for
  $k=0, 1, \dots, n$ and filling the gaps between these points with
  $n$ Brownian bridges, independent of $X$ and of each other.  Then we
  denote the distribution of $Y$ by~$\hat\mu_n$.
\end{definition}

\begin{lemma}\label{L:Pi-and-hat}
  Let $\mu_n$ and $\nu_n$ be probability measures on $\R^I$ with
  $\mu_n \ll \nu_n$.  Then $\hat\mu_n \ll \hat\nu_n$ with
  \begin{equation*}
    \frac{d\hat\mu_n}{d\hat\nu_n} = \frac{d\mu_n}{d\nu_n}\circ\Pi
  \end{equation*}
  on $\rC\bigl([0,1],\R\bigr)$.
\end{lemma}

\begin{proof}
  Let $\psi = \frac{d\mu_n}{d\nu_n}$.  Using substitution we get
  \begin{equation*}
    \E_{\hat\mu_n}\bigl(f\circ\Pi\bigr)
    = \int_{\R^I} f \,d\mu_n
    = \int_{\R^I} f \psi \,d\nu_n
    = \E_{\hat\nu_n}\bigl(f\circ\Pi \cdot \psi\circ\Pi \bigr)
  \end{equation*}
  for all integrable $f\colon \R^I\to \R$.  Since, conditioned on
  the value of $\Pi$, the distributions $\hat\mu_n$ and $\hat\nu_n$
  are the same, we can use the tower property to get
  \begin{align*}
    \hat\mu(A)
    &= \E_{\hat\mu_n}\bigl(\E_{\hat\mu_n}(1_A|\Pi)\bigr)
    = \E_{\hat\mu_n}\bigl(\E_{\hat\nu_n}(1_A|\Pi)\bigr) \\
    &= \E_{\hat\nu_n}\bigl(\E_{\hat\nu_n}(1_A|\Pi) \cdot \psi\circ\Pi \bigr)
    = \E_{\hat\nu_n}\bigl(1_A \psi\circ\Pi \bigr)
  \end{align*}
  for every measurable set~$A$.  This shows that $\psi\circ\Pi$ is the
  required density.
\end{proof}

\begin{corollary}\label{C:other}
  Let $\mu$ be the stationary distribution of the SPDE~\eqref{E:SPDE}
  on $\rC\bigl([0,1],\R\bigr)$.  Let $\mu_n$ be the stationary
  distribution of the finite element equation~\eqref{E:FE-SDE}
  on~$\R^I$.  Let $\CL$ be negative and assume $f=F'$ where $F\in
  \rC^2(\R)$ is bounded from above with bounded second derivative.
  Then
  \begin{equation*}
    \bigl\| \mu - \hat \mu_n \bigr\|_{\mathrm{TV}}
    = \CO\bigl(\frac1n\bigr)
  \end{equation*}
  as $n\to\infty$.
\end{corollary}

\begin{proof}
  Let $\nu$ be the stationary distribution of the linear
  SPDE~\eqref{E:linear} on $\rC\bigl([0,1],\R\bigr)$ and let $\nu_n$
  be the stationary distribution of the linear finite element
  equation~\eqref{E:FE-linear} on~$\R^I$.  By construction of the
  process $U$ in the third statement of lemma~\ref{L:C-exact} and by
  the Markov property for Brownian bridges, the distribution of $U$
  between the grid points, conditioned on the values at the grid
  points, coincides with the distribution of $n$ independent Brownian
  bridges.  By lemma~\ref{L:exact} the distribution of $U$ on the grid
  points is given by~$\nu_n$.  Thus we have $\nu = \hat\nu_n$.  Using
  this equality and lemma \ref{L:Pi-and-hat} we find
  \begin{equation*}
    \bigl\| \mu - \hat\mu_n \bigr\|_{\mathrm{TV}}
    = \E_\nu\Bigl| \frac{d\mu}{d\nu} - \frac{d\hat\mu_n}{d\nu} \Bigr|
    = \E_\nu\Bigl| \frac{d\mu}{d\nu} - \frac{d\mu_n}{d\nu_n} \circ \Pi \Bigr|.
  \end{equation*}
  Now we are in the situation of equation~\eqref{E:proof-start} and
  the proof of theorem~\ref{T:result} applies without further changes.
\end{proof}

\section{Examples}
\label{S:example}

To illustrate that the suggested finite element method is a concrete
and implementable scheme, this section gives two examples for the
finite element discretisation of SPDEs, both in the context of infinite
dimensional sampling problems.

For the first example, for $c>0$, consider the SPDE
\begin{equation*}
  \d_t u (t,x) = \d_x^2 u(t,x) - c^2 u(t,x) + \sqrt{2} \d_t w(t,x)
  \qquad \forall (t,x) \in \R_+ \times (0,1)
\end{equation*}
with Robin boundary conditions
\begin{equation} \label{E:ex1-bc}
  \d_xu(t,0) = c u(t,0), \quad \d_x u(t,1) = - c u(t,1)
  \qquad \forall t\in \R_+,
\end{equation}
where $\d_t w$ is space-time white noise.  From~\citet{HaiStuaVo07} we
know that the stationary distribution of this SPDE on $\rC\bigl([0,1],
\R\bigr)$ coincides with the distribution of the process~$X$ given by
\begin{equation}\label{E:ex1-SDE}
  \begin{split}
  dX_\tau &= - c X_\tau \,d\tau + dW_\tau \qquad \forall \tau\in[0,1] \\
  X_0 &\sim \CN\bigl(0, \frac{1}{2c} \bigr),
  \end{split}
\end{equation}
where the time~$\tau$ in the SDE plays the r\^ole of the space~$x$ in
the SPDE.  In the framework of section~\ref{S:SPDE}, the boundary
conditions~\eqref{E:ex1-bc} correspond to the case $\alpha_0 =
\alpha_1 = c$ and $\beta_0 = \beta_1 = 1$.  Since $\beta_i \neq 0$, we
need to include both boundary points in the finite element
discretisation and thus have $I=\{ 0, 1, \ldots, n \}$ and $\R^I \cong
\R^{n+1}$.  The matrix $L^{(n)}$ is given by
\begin{equation*}
  L^{(n)} = \frac{1}{\DX}
  \begin{pmatrix}
    -1 - c \DX & 1 & & & \\
    1 & -2 & 1 & & \\
    & 1 & -2 & 1 & \\
    & & 1 & -2 & 1 \\
    & & & 1 & -1 - c \DX
  \end{pmatrix}
  \in \R^{(n+1)\times(n+1)},
\end{equation*}
where the middle rows are repeated along the diagonal to obtain
tridiagonal $(n+1)\times(n+1)$-matrices.  Similarly, the mass matrix
$M$ is given by
\begin{equation*}
  M = \frac{\DX}{6}
  \begin{pmatrix}
    2 & 1 & & & \\
    1 & 4 & 1 & & \\
    & 1 & 4 & 1 & \\
    & & 1 & 4 & 1 \\
    & & & 1 & 2
  \end{pmatrix}
  \in \R^{(n+1)\times(n+1)}
\end{equation*}
where, again, the middle rows are repeated along the diagonal.
Finally, it transpires that the discretised drift for this example is
given by $f_n(u) = - c Mu$.  By lemma~\ref{L:preconditioning} the
$n+1$ dimensional SDEs
\begin{equation*}
  dU_t = M^{-1} L^{(n)} U_t \,dt - c U_t \,dt + \sqrt{2} M^{-1/2} \,dW_t
\end{equation*}
and
\begin{equation*}
  dU_t = L^{(n)} U_t \,dt - c M U_t \,dt + \sqrt{2}\,dW_t,
\end{equation*}
where $W$ is an $(n+1)$-dimensional Brownian motion, both have the
same stationary distribution and this stationary distribution
converges to the distribution of the process~$X$
from~\eqref{E:ex1-SDE} in the sense given in theorem~\ref{T:result}
and corollary~\ref{C:other}.

As a second example, consider the SPDE
\begin{equation*}
  \d_t u (t,x) = \d_x^2 u(t,x) - \bigl(g g' + \frac12 g'')(u) + \sqrt{2} \d_t w(t,x)
  \qquad \forall (t,x) \in \R_+ \times (0,1)
\end{equation*}
with Dirichlet boundary conditions
\begin{equation*}
  u(t,0) = u(t,1) = 0
  \qquad \forall t\in \R_+,
\end{equation*}
where $g\in\rC^3\bigl(\R, \R\bigr)$ with bounded derivatives $g'$,
$g''$ and $g'''$.  From~\citet{HaiStuaVo07} we know that the
stationary distribution of this SPDE on $\rC\bigl([0,1], \R\bigr)$
coincides with the conditional distribution of the process~$X$ given
by
\begin{equation}\label{E:ex2-SDE}
  \begin{split}
  dX_\tau &= g(X_\tau) \,d\tau + dW_\tau \qquad \forall \tau\in[0,1] \\
  X_0 &= 0,
  \end{split}
\end{equation}
conditioned on $X_1 = 0$.

Since we have Dirichlet boundary conditions, the boundary points in
the finite element discretisation are not included: we have $I=\{1,
2, \ldots, n-1\}$ and $\R^I \cong \R^{n-1}$.  The matrices $L^{(n)}$
and $M$ are given by
\begin{equation*}
  L^{(n)} = \frac{1}{\DX}
  \begin{pmatrix}
    -2 & 1 & \\
    1 & -2 & 1 \\
    & 1 & -2
  \end{pmatrix}
  \in \R^{(n-1)\times(n-1)},
\end{equation*}
and
\begin{equation*}
  M = \frac{\DX}{6}
  \begin{pmatrix}
    4 & 1 & \\
    1 & 4 & 1 \\
    & 1 & 4
  \end{pmatrix}
  \in \R^{(n-1)\times(n-1)}
\end{equation*}
where, again, the middle rows are repeated along the diagonal to
obtain matrices of the required size.  The discretised drift~$f_n$ can
be computed from~\eqref{E:FE-f}; if an analytical solution is not
available, numerical integration can be used.  By the assumptions
on~$g$, the function $F = - \frac12 (g^2 + g')$ satisfies the
conditions of theorem~\ref{T:result}.  Thus, the stationary
distributions of the $(n-1)$-dimensional SDEs
\begin{equation*}
  dU_t
  = M^{-1} L^{(n)} U_t \,dt
    + M^{-1} f_n(U_t) \,dt
    + \sqrt{2} M^{-1/2}\, dW_t
\end{equation*}
and
\begin{equation*}
  dU_t
  = L^{(n)} U_t \,dt
    + f_n(U_t) \,dt
    + \sqrt{2}\, dW_t
\end{equation*}
coincide and converge to the conditional distribution of $X$
from~\eqref{E:ex2-SDE}, conditioned on~$X_1 = 0$.


\bibliographystyle{plainnat}
\bibliography{fe-stat}

\begin{thebibliography}{19}
\providecommand{\natexlab}[1]{#1}
\providecommand{\url}[1]{\texttt{#1}}
\expandafter\ifx\csname urlstyle\endcsname\relax
  \providecommand{\doi}[1]{doi: #1}\else
  \providecommand{\doi}{doi: \begingroup \urlstyle{rm}\Url}\fi

\bibitem[Arnold(1974)]{Arn74}
Ludwig Arnold.
\newblock \emph{Stochastic Differential Equations: Theory and Applications}.
\newblock John Wiley \& Sons, 1974.

\bibitem[Beskos et~al.(2008)Beskos, Roberts, Stuart, and Voss]{BeRoStuaVo08}
Alexandros Beskos, Gareth~O. Roberts, Andrew~M. Stuart, and Jochen Voss.
\newblock {MCMC} methods for diffusion bridges.
\newblock \emph{Stochastics and Dynamics}, 8\penalty0 (3):\penalty0 319--350,
  2008.
\newblock \doi{10.1142/S0219493708002378}.

\bibitem[Brenner and Scott(2002)]{BreSco02}
S.~C. Brenner and L.~R. Scott.
\newblock \emph{The Mathematical Theory of Finite Element Methods}.
\newblock Springer, second edition, 2002.

\bibitem[Cacciapuoti and Finco(2010)]{CaccFin10}
Claudio Cacciapuoti and Domenico Finco.
\newblock Graph-like models for thin waveguides with {R}obin boundary
  conditions.
\newblock \emph{Asymptotic Analysis}, 70\penalty0 (3--4):\penalty0 199--230,
  2010.
\newblock \doi{10.3233/ASY-2010-1014}.

\bibitem[Da~Prato and Zabczyk(1992)]{ZDP}
Giuseppe Da~Prato and Jerzy Zabczyk.
\newblock \emph{Stochastic Equations in Infinite Dimensions}, volume~44 of
  \emph{Encyclopedia of Mathematics and its Applications}.
\newblock Cambridge University Press, 1992.
\newblock ISBN 0-521-38529-6.

\bibitem[Da~Prato and Zabczyk(1996)]{DaPraZa96}
Giuseppe Da~Prato and Jerzy Zabczyk.
\newblock \emph{Ergodicity for Infinite-Dimensional Systems}, volume 229 of
  \emph{London Mathematical Society Lecture Note Series}.
\newblock Cambridge University Press, 1996.

\bibitem[Fernique(1970)]{Fe70}
Xavier Fernique.
\newblock Int{\'e}grabilit{\'e} des vecteurs gaussiens.
\newblock \emph{Comptes rendus hebdomadaires des s{\'e}ances de l'Acad{\'e}mie
  des sciences, s{\'e}rie A}, 270:\penalty0 1698--1699, 1970.

\bibitem[Gy{\"o}ngy and Millet(2009)]{MR2465711}
Istv{\'a}n Gy{\"o}ngy and Annie Millet.
\newblock Rate of convergence of space time approximations for stochastic
  evolution equations.
\newblock \emph{Potential Analysis}, 30\penalty0 (1):\penalty0 29--64, 2009.
\newblock \doi{10.1007/s11118-008-9105-5}.

\bibitem[Hairer et~al.(2005)Hairer, Stuart, Voss, and Wiberg]{HaiStuaVoWi05}
Martin Hairer, Andrew~M. Stuart, Jochen Voss, and Petter Wiberg.
\newblock Analysis of {SPDEs} arising in path sampling, part~{I}: The
  {G}aussian case.
\newblock \emph{Communications in Mathematical Sciences}, 3\penalty0
  (4):\penalty0 587--603, 2005.

\bibitem[Hairer et~al.(2007)Hairer, Stuart, and Voss]{HaiStuaVo07}
Martin Hairer, Andrew~M. Stuart, and Jochen Voss.
\newblock Analysis of {SPDEs} arising in path sampling, part~{II}: The
  nonlinear case.
\newblock \emph{Annals of Applied Probability}, 17\penalty0 (5):\penalty0
  1657--1706, 2007.
\newblock \doi{10.1214/07-AAP441}.

\bibitem[Hairer et~al.(2009)Hairer, Stuart, and Voss]{HaiStuaVo09}
Martin Hairer, Andrew~M. Stuart, and Jochen Voss.
\newblock Sampling conditioned diffusions.
\newblock In \emph{Trends in Stochastic Analysis}, volume 353 of \emph{London
  Mathematical Society Lecture Note Series}, pages 159--186. Cambridge
  University Press, 2009.
\newblock ISBN 9780521718219.

\bibitem[Hausenblas(2008)]{Hau08}
Erika Hausenblas.
\newblock Finite element approximation of stochastic partial differential
  equations driven by {P}oisson random measures of jump type.
\newblock \emph{SIAM Journal on Numerical Analysis}, 46\penalty0 (1):\penalty0
  437--471, 2008.
\newblock \doi{10.1137/050654141}.

\bibitem[Iscoe et~al.(1990)Iscoe, Marcus, McDonald, Talagrand, and
  Zinn]{IscMaMcDoTaZi90}
I.~Iscoe, M.~B. Marcus, D.~McDonald, M.~Talagrand, and J.~Zinn.
\newblock Continuity of {$l^2$}-valued {O}rnstein-{U}hlenbeck processes.
\newblock \emph{The Annals of Probability}, 18\penalty0 (1):\penalty0 68--84,
  1990.

\bibitem[Jentzen(2011)]{Jen11}
Arnulf Jentzen.
\newblock Higher order pathwise numerical approximations of {SPDE}s with
  additive noise.
\newblock \emph{SIAM Journal on Numerical Analysis}, 49\penalty0 (2):\penalty0
  642--667, 2011.
\newblock \doi{10.1137/080740714}.

\bibitem[Johnson(1990)]{Jo90}
C.~Johnson.
\newblock \emph{Numerical Solution of Partial Differential Equations by the
  Finite Element Method}.
\newblock Cambridge University Press, 1990.

\bibitem[Lancaster and Rodman(1995)]{Lancaster-Rodman95}
Peter Lancaster and Leiba Rodman.
\newblock \emph{Algebraic {R}iccati Equations}.
\newblock Clarendon Press, Oxford, 1995.

\bibitem[Millet and Morien(2005)]{MR2147242}
Annie Millet and Pierre-Luc Morien.
\newblock On implicit and explicit discretization schemes for parabolic {SPDE}s
  in any dimension.
\newblock \emph{Stochastic Processes and their Applications}, 115\penalty0
  (7):\penalty0 1073--1106, 2005.
\newblock \doi{10.1016/j.spa.2005.02.004}.

\bibitem[Walsh(2005)]{Wa05}
J.~B. Walsh.
\newblock Finite element methods for parabolic stochastic {PDE}'s.
\newblock \emph{Potential Analysis}, 23\penalty0 (1):\penalty0 1--43, 2005.
\newblock \doi{10.1007/s11118-004-2950-y}.

\bibitem[Zabczyk(1989)]{Za89}
Jerzy Zabczyk.
\newblock Symmetric solutions of semilinear stochastic equations.
\newblock In Giuseppe Da~Prato and L.~Tubaro, editors, \emph{Stochastic Partial
  Differential Equations and Applications {II}}, volume 1390 of \emph{Lecture
  Notes in Mathematics}, pages 237--256. Springer, 1989.
\newblock \doi{10.1007/BFb0083952}.

\end{thebibliography}

\end{document}